\documentclass[12pt]{amsart}
\usepackage[utf8]{inputenc}
\usepackage{amsmath,amsthm,amscd,amssymb,bbm,graphicx}
\usepackage{floatflt,setspace, wrapfig}
\usepackage{graphics,multicol}
\usepackage{mathrsfs} 
\usepackage{epstopdf}
\usepackage[letterpaper,
            bindingoffset=0.2in,
            left=1in,
            right=1in,
            top=1in,
            bottom=1in,
            footskip=0.5in]{geometry}
\usepackage{todonotes}
\usepackage{bbm}
\usepackage{xcolor}

\newtheorem{thrm}{Theorem}[section]

\newtheorem*{cor}{Corollary}

\newtheorem{lem}{Lemma}[section]

\newtheorem{remark}{Remark}[section]

\newcommand{\ind}{\mathbbm{1}}

\newcommand{\ball}{B_{\rho}}

\DeclareMathOperator{\diam}{diam}

\providecommand{\abs}[1]{\lvert \, #1 \, \rvert}
\providecommand{\Abs}[1]{\biggl\lvert \, #1 \, \biggr \rvert}

\newcommand{\bea}[1]{\begin{eqnarray}\label{#1}}
\newcommand{\eea}{\end{eqnarray}}

\allowdisplaybreaks[4]

\usepackage{subfiles}

\title{Return times distribution of expanding maps}

\begin{document}
 \maketitle
\authors{Nicolai T A Haydn\footnote{Department of Mathematics, University of Southern California,
Los Angeles, 90089-2532. E-mail: {\tt \email{nhaydn@usc.edu}}. Supported by Simons Foundation
{\it Collaboration Grants for Mathematicians}: ID 526571}}


\maketitle

\pagenumbering{roman}

\section*{Abstract}
We consider expanding systems with invariant measures that are uniformly expanding everywhere
except on a small measure set and show that the limiting statistics of hitting times for 
zero measure sets are compound Poisson provided the limits for the cluster size distributions
exist. This extends previous results from neighbourhoods around single points to neighbourhoods
around zero measure sets. The assumptions require the correlations to decay at least 
polynomially and the non-uniformly expanding part of the iterates of the map also has to satisfy some decay condition.
We also require  some 
regularity conditions around the limiting zero measure target set.\footnote{Keywords:
Entry times distribution, Compound Poisson distribution, Expanding maps, Parabolic interval map}


\pagenumbering{arabic}

\section{Introduction}

Limiting return times distributions have been studied already by Doeblin for the Gauss map,
but serious broad interest developed only around 1990 in particular with a paper by 
Pitskel~\cite{Pit91} where he showed, using generating functions, that for Axiom~A maps the limiting distribution
at almost all points are Poissonian if the shrinking target sets are cylinder sets.
He also found that at periodic points the return times also have a geometric
component which gives rise to a compound Poisson distribution. 
A similar result was shown in~\cite{Hirata1} for the entry time using the Laplace
transform and then it was argued that this extends to Poisson distributed limiting
higher return times using the weak mixing property.
For harmonic maps the Poisson limiting distribution was shown in~\cite{DGS} using the Chen-Stein method.
Later, Galves and Schmitt~\cite{GS97} came up with a very effective method
to show exponentially distributed entry times in a symbolic setting for $\psi$-mixing 
measures and also provided error terms. This method was then carried further 
by Abadi and others (e.g.~\cite{Aba01, HP10,AV08,AS11}) where it was shown that
for symbolic systems the limiting return times are almost everywhere Poisson 
distributed. Similar results for metric balls for non-uniformly hyperbolic systems
 where shown in~\cite{CC13,HW14,PS16} even providing speeds of convergence
 outside a small set.  A closer look at periodic orbits was done in~\cite{HV09} where
 it was found that in a symbolic setting one always gets a compound Poisson
 distrution. Later there was found to be a dichotomy between non-periodic and 
 periodic points where non-periodic points always have Poisson distributed hitting
 times and periodic points lead to compound Poisson distributions
 (see e.g.~\cite{FFT13,FHN14,FFM17} and also~\cite{B16}).
 In~\cite{HV20} such concepts were generalised to the situation when the 
 limiting set is not a single point any longer but can be some arbitrary null
 set. In this case the limiting compound Poisson distribution can be 
 quite general and is not any longer restricted to be P\'olya-Aeppli.
 In that instance it was used to get limiting results on synchonised systems.
 Here we show a similar argument for expanding maps only which we then
 apply to interval maps and look in particular at the limiting hitting times
 distribution at the parabolic point where we recover a result of~\cite{FFTV18}
 where an ad hoc method was used to get the right scaling which is 
 different from the standard Kac scaling and which follows in a natural
 way from our setup.

\section{Expanding Maps and Limiting Distributions} \label{expanding.maps}

Let $\Omega$ be a metric space
and $T:\Omega\to \Omega$ be a measurable map and $\mu$ a $T$-invariant probability
measure. Here we assume that $T$ is differentiable and expanding, i.e.\ $|DT(x)|\ge 1$ 
for all $x$. If $U\subset \Omega$ is a subset then put (as in~\cite{HV20,HV09,Rou14})
$$
Z_U^N=\sum_{j=0}^{N-1}\ind_U\circ T^j
$$
for the hit-counting function over the time interval $N$. In particular, if 
 $\Gamma\subset\Omega$ has zero $\mu$-measure
and for $\rho>0$ we denote by $B_\rho(\Gamma)=\bigcup_{x\in\Gamma}B_\rho(x)$ its $\rho$-neighbourhood,
then for suitable orbit lengths $N(\rho)$ we want to get that the distribution of
$\zeta_{B_\rho(\Gamma)}^{N(\rho)}$ converges to a non-degenerate probability distribution
as $\rho\to0$. In the classical setting $\Gamma$ is usually chosen to be a simple 
point  $x$ and for $N$ one then choses the Kac scaling $t/\mu(B_\rho(x))$ where
$>0$ is a parameter (as was done for instance
 in~\cite{Aba01,AS11, CC13,FFM17,GS97,HV09,HV20, HW14, Hirata1, PS16,Pit91,Rou14}).
 In that case it is known that for a variety of systems 
the limiting distribution is a compound Poisson distribution. Here we present a 
general scheme for expanding maps to obtain such results which might also 
result in a non-standard scaling for $N(\rho)$.

Let us recall that an integer valued random variable $W$ is compound Poisson distributed
  if there are i.i.d.\ $\mathbb{N}$-valued random variables $Y_j\ge 1$, $j=1,2,\dots$,
  and an independent Poisson distributed random variable $P$ so
  that $W=\sum_{j=1}^PY_j$. The Poisson distribution $P$ describes the
  distribution of what usually is referred to as  clusters whose sizes are then described by the 
  values of the random variables $Y_j$
  whose probability densities are given by $\lambda_\ell=\mathbb{P}(Y_j=\ell)$, $\ell=1,2,\dots$.
In particular
$$
\mathbb{P}(W=k)=\sum_{\ell=1}^k\mathbb{P}(P=\ell)\mathbb{P}(S_\ell=k),
$$
where $S_\ell=\sum_{j=1}^\ell Y_j$ and $P$ is Poisson distributed with parameter
$t$, i.e.\ $\mathbb{P}(P=\ell)=e^{-t}t^\ell/\ell!$,  and by Wald's equation (se e.g.~\cite{Durr})
$\mathbb{E}(W)=t\mathbb{E}(Y_j)$.
We say a probability measure $\tilde\nu$ on $\mathbb{N}_0$ is compound
Poisson distributed with parameters $t$ and $\lambda_\ell$, $\ell=1,2,\dots$
if it has the same distribution as $W$.

We recover the Poisson distribution $W=P$ in the special case $Y_1=1$ and  $\lambda_\ell=0$
for all $\ell=2,3,\dots$. 
 
 An important non-trivial compound Poisson distribution is the P\'olya-Aeppli
 distribution which happens when  the $Y_j$ are geometrically distributed,
 that is $\lambda_\ell=\mathbb{P}(Y_\ell)=(1-\vartheta)\vartheta^{\ell-1}$ for $\ell=1,2,\dots$, 
 for some $\vartheta\in(0,1)$.
 In this case (cf.~\cite{Nu08,HV09, HV20})
 $$
 \mathbb{P}(W=k)=e^{-t}
\sum_{j=1}^k\vartheta^{k-j}(1-\vartheta)^j\frac{s^j}{j!}\binom{k-1}{j-1}
$$
and in particular $\mathbb{P}(W=0)=e^{-t}$.
In the case of $p=0$ this reverts back to the straight Poisson distribution.

\section{Assumptions and main result}\label{assumptions}

We will make assumptions that are designed to cover a large class of dynamical
systems acting on manifolds. The formulation follows~\cite{HV20} and a similar 
setting was used in a random setting in~\cite{HRY20}.

Assume there exist $R>0$ such that for every $n\in\mathbb{N}$ there are finitely many
 $y_k\in \Omega$ so that 
$\Omega\subset\bigcup_k B_{R}(y_k)$, 
where $B_{R}(y_k)$ is the $R$-disk centered at $y_k$.
Denote by $\zeta_{\psi,k}=\psi(B_{R}(y_k))$ where $\psi\in \hat{\mathscr{I}}_n$
 and $\hat{\mathscr{I}}_n$ denotes the  inverse branches of $T^n$. 
 We call such a $\zeta$ an $n$-cylinder.
Then there exists a constant $\mathfrak{N}$ so that the number of overlaps
$N_{\psi,k}=|\{\zeta_{\psi',k'}: \zeta_{\psi,k}\cap\zeta_{\psi',k'}\not=\varnothing,
\psi'\in\hat{\mathscr{I}}_n\}|$
is bounded by $\mathfrak{N}$ for all $\psi\in \hat{\mathscr{I}}_n$ and for all $k$ and $n$. 
We use $|\cdot|$ to denote cardinality which makes  $N_{\psi,k}$ count the number of
$n$-cylinders that overlap with the given $n$-cylinder $\zeta_{\psi,k}$.
The uniform boundedness follows from the fact that
$N_{\psi,k}$ equals $|\{k': B_{R}(y_k)\cap B_{R}(y_{k'})\not=\varnothing\}|$ ($|\cdot|$ is cardinality)
which is uniformly bounded by some constant $\mathfrak{N}$ independently of $n$.

For simplicity of notation let us index the $n$-cylinders $\zeta_\varphi$ by $\varphi$, where 
$\varphi\in \{(\psi,k): \psi\in\hat{\mathscr{I}}_n, k\}$ and denote by $\mathscr{I}_n$ the union over
all such $\varphi$.

Let us denote by 
$J_n=\frac{dT^n\mu}{\, d\mu}$
the Jacobian of the map $T^n$ with respect to the measure $\mu$.

 For $L\in \mathbb{N}$ put
$Z_{U}^L=\sum_{j=0}^{L-1}\ind_U\circ T^j$ for the hit counting function over a time interval of length $L$.
Let $\Gamma\subset\Omega$ be a zero $\mu$-measure set and for $\rho>0$ denote by 
$B_\rho(\Gamma)=\bigcup_{x\in\Gamma}B_\rho(x)$ 
its $\rho$-neighbourhood.
Then we shall make the following assumptions:\\
(I) There exists a decay function $\mathcal{C}(k)$ so that 
$$
\left|\int_\Omega G(H\circ T^k)\, d\mu
-\mu(G)\mu(H)\right|
\le \mathcal{C}(k)\|G\|_{Lip}\|H\|_\infty\qquad\forall k\in\mathbb{N},
$$
for every $H\in L^\infty(\Omega,\mathbb{R})$  for every $G\in Lip(\Omega,\mathbb{R})$.\\
Then, we need some geometric assumptions:\\
(II) We assume that there is a set $\mathcal{G}_n\subset \{\zeta_\varphi: \varphi\in\mathscr{I}_n\}$ (set of good $n$-cylinders)
so that \\
(i) $\mu(G_n^c)\lesssim n^{-\mathfrak{g}}$ for some $\mathfrak{g}>0$, 
where $G_n=\bigcup_{\zeta\in\mathcal{G}_n}\zeta$.\\
(ii) (Distortion) We  require that 
$\frac{J_n(x)}{J_n(y)}=\mathcal{O}(\mathfrak{D}(n))$ for all $x,y\in\zeta$  for all 
$\zeta\in \mathcal{G}_n$ $n$-cylinders and all $n$, where $\mathfrak{D}$ is a non-decreasing 
function which below we assume to be $\mathfrak{D}(n)=\mathcal{O}(n^{\mathfrak{d}})$ for 
some $\mathfrak{d}\ge0$. \\
(iii) (Contraction) There exists a function $\delta(n)\to0$ which decays at least summably polynomially, i.\,e.\, 
$\delta(n) = \mathcal{O}(n^{-\mathfrak{k}})$ with $\mathfrak{k} > 1$, so that 
$\diam\zeta\le \delta(n)$ for all $n$-cylinder $\zeta\in \mathcal{G}_n$ and all $n$.\\\
(iv) 
$$
\mathfrak{G}_{\rho,L}=\sum_{n=L}^\infty
\frac{\mu(G_n^c\cap B_\rho(\Gamma)\cap T^{-n}\mathcal{V}_\rho^L)}{\mu(\mathcal{V}_\rho^L)}\longrightarrow0
$$
as $\rho\to0, L\to\infty$, where $\mathcal{V}_\rho^L=\{x\in\Omega: Z_{B_\rho(\Gamma)}^L(x)\ge1\}$.\\
(III) (Dimension estimate) There exist  $0<d_0<d_1<\infty$ such that 
$\rho^{d_1}\le\mu(B_\rho(\Gamma)) \le \rho^{d_0}$.\\
(IV) (Annulus condition) Assume that for some $\xi\ge\beta>0$:
$$
\frac{\mu(B_{\rho+r}(\Gamma)\setminus B_{\rho-r}(\Gamma))}{\mu(\ball(\Gamma))} 
= \mathcal{O}(\frac{r^\xi}{\rho^\beta})
$$
for every $r < \rho$. 

In the standard setting $\Gamma$ typically is a single point. In the general case however the 
dimension condition~(III) and annulus condition~(IV) put obvious constraints on the 
sets $\Gamma$, which obviously cannot be a dense set for instance, but can contain 
curve segments in the case of piecewise expanding maps in $2$ dimensions (as e.g.\ in Section~\ref{product.interval.maps}).

Here and in the following we use the notation $x_n\lesssim y_n$ for $n=1,2, \dots$, to mean that
there exists a constant $C$ so that $x_n<Cy_n$ for all $n$.
 As before let
$T:\Omega\circlearrowleft$ and $\mu$ a $T$-invariant probability measure
on $\Omega$. For a subset $U\subset\Omega$ we put $X_i=\ind_U\circ T^i$ and define
$$
Z^L=Z^L_U=\sum_{i=0}^{L-1}X_i
$$
where $L$ is a (large) positive integer. For $\Gamma\subset \Omega$ of zero measure we put
\begin{equation}\label{LLL}
 \lambda_\ell=\lim_{L\to\infty}\lambda_\ell(L),
 \end{equation}
where
$$
\lambda_\ell(L)=\lim_{\rho\to0}\frac{\mathbb{P}(Z^L_{B_\rho(\Gamma)}=\ell)}{
\mathbb{P}(Z^L_{B_\rho(\Gamma)}\ge1)}.
$$

Let us now formulate our main result.

\begin{thrm}\label{main.theorem}
Assume that the map $T: \Omega\to \Omega$ satisfies the assumptions (I)--(IV) where $\mathcal{C}(k)$ decays at
least polynomially with power $\mathfrak{p}>(\frac\beta\eta+d_1)\frac{1+\mathfrak{d}}{d_0}$
and $\mathfrak{k}>\frac{1+\mathfrak{d}}{d_0}$.
Let $\Gamma\subset \Omega$ be a zero measure set and $\lambda_\ell$ the corresponding quantity as defined in (\ref{LLL}).

Then
$$
\lim_{L\to\infty}\lim_{\rho\to0}\mathbb{P}(Z_{B_\rho(\Gamma)}^N=k)=\nu(\{k\}),
$$
where $\nu$ is the compound Poisson distribution for the parameters
$t$, $\lambda_\ell$ and $N=N(L,\rho)=\frac{tL}{\mathbb{P}(Z_{B_\rho(\Gamma)}^L\ge1)}$.
\end{thrm}

The proof of  Theorem~\ref{main.theorem} is given in  Section~\ref{proof.theorem}.
In the following section we will express the parameters $\lambda_\ell$ in terms of the limiting
return times distribution.


\section{Return times and Kac's scaling}\label{return.times}
As before let $U\subset\Omega$ be a subset of $\Omega$ so that $\mu(U)>0$, 
then define the first entry/return time
 $\tau_U$ by $\tau_U(x)=\min\{j\ge1: T^j\in U\}$ which by Poincar\'e's recurrence theorem
 is finite almost everywhere. Similarly we get higher order
 returns by defining recursively $\tau_U^\ell(x)=\tau_U^{\ell-1}+\tau_U(T^{\tau_U^{\ell-1}}(x))$
 with $\tau_U^1=\tau_U$. We also write $\tau_U^0=0$ on $U$.
 
For  $L$ some large number we then put
$\hat\alpha_\ell(L,U)=\mu_U(\tau_U^{\ell-1}<L)$ for $\ell=1,2,\dots$, ($\hat\alpha_1=1$)
and also
 $\alpha_\ell(L,U)=\mu_{U}(\tau_{U}^{\ell-1}< L\le \tau_{U}^\ell)=\hat\alpha_\ell(L,U)-\hat\alpha_{\ell+1}(L,U)$
for $\ell=1,2,\dots$,  where
$\mu_{U}$ is the induced measure on $U$ given by
$\mu_{U}(A)=\mu(A\cap U)/\mu(U),\forall A\subset\Omega$.

Now let $U_n\subset \Omega$, $n=1,2,\dots$, be a nested sequence of sets and put
$\Gamma=\bigcap_nU_n$. We assume that $\mu(U_n)\to 0$ as $n\to\infty$ which of course implies
that $\mu(\Gamma)=0$. Assume that the limits $\hat\alpha_\ell(L)=\lim_{n\to\infty}\hat\alpha_\ell(L,U_n)$
exist for all $L$ large. By monotonicity the limits $\hat\alpha_\ell =\lim_{L\to\infty}\hat\alpha_\ell(L)$ exist.
This then implies the existence of the limits
  $\alpha_\ell(L)=\lim_{n\to\infty}\alpha_\ell(L,U_n)$, $\ell=1,2,\dots$,  for all $L$ large enough
  and similarly we then get $\alpha_\ell=\lim_{L\to\infty}\alpha_\ell(L)$,
for $\ell=1,2,\dots$.
 In the special case $\ell=1$ the value
$\alpha_1=\lim_{L\to\infty}\lim_{n\to\infty}\mu_{U_n}(L\le \tau_{U_n})$ is called
the extremal index.

\begin{lem}\label{theorem.lambda}\cite{HV20}
Assume $\sum_\ell\ell^2\alpha_\ell<\infty$. If $\alpha_1>0$ then
$$
\lambda_k=\frac{\alpha_k-\alpha_{k+1}}{\alpha_1}.
$$
In particular the limit defining $\lambda_k$ exists.
\end{lem}

\noindent This lemma in particular implies that the expected length of the clusters is given by
$$
\sum_{k=1}^\infty k\lambda_k=\frac{1}{{\alpha_1}}\sum_{k=1}^{\infty}k(\alpha_k-\alpha_{k+1})=\frac1{\alpha_1}
$$
provided $\alpha_1$ is positive. Also notice that since  $\lambda_k\ge0$ one obtains that
 $\alpha_1\ge\alpha_2\ge\alpha_3\ge\cdots$ is
a decreasing sequence. Moreover 
$\lambda_k=\alpha_k\forall k$ only when both are geometrically distributed, i.e.\
$\lambda_k=\alpha_k=\alpha_1(1-\alpha_1)^k$ which results in a Polya-Aeppli
compound Poisson distribution.

Let us consider the entry time $\tau_U(x)$ where $x\in\Omega$.

\begin{lem}
Let $\mu$ be a $T$-invariant probability measure and denote by $\mathbb{P}=\mathbb{P}_\mu$
the probability with respect to it.

Then for $U\subset\Omega$, $\mu(U)>0$, one has
$$
\mathbb{P}(\tau_U< L)=\mu(U)\sum_{j=1}^{L}\alpha_1(j,U).
$$
If moreover $U_n\subset\Omega$ is a nested sequence so that $\mu(U_n)\to0$ as $n\to\infty$
and so that the limit $\alpha_1=\lim_{L\to\infty}\lim_{n\to\infty}\alpha_1(L,U_n)$ exists and satisfies
$\alpha_1>0$, then
$$
\lim_{L\to\infty}\lim_{n\to\infty}\frac{\mathbb{P}(\tau_{U_n}<L)}{L\mu(U_n)}=\alpha_1.
$$
\end{lem}

\begin{proof}
We have
\begin{eqnarray*}
\mathbb{P}(\tau_U<L)
&=&\mathbb{P}(Z^L\ge1)\\
&=&\sum_{j=0}^{L-1}\mathbb{P}(Z^j\ge1,T^{-j}U, \tau_U\circ T^j\ge L-j)\\
&=&\sum_{j=0}^{L-1}\mathbb{P}(U, \tau_U\ge L-j)\\
&=&\mu(U)\sum_{k=1}^{L}\alpha_1(k,U).
\end{eqnarray*}
The second part of the statement now follows if $\alpha_1>0$.
\end{proof}

\begin{remark}
If $\alpha_1>0$ then it now follows from the two previous lemmata and their proofs that
$$
\lim_{L\to\infty}\lim_{n\to\infty}\frac{\mathbb{P}(\tau_{U_n}^\ell< L\le \tau_{U_n}^{\ell+1})}{L\mu(U_n)}
=\alpha_1\lambda_\ell
$$
for $\ell=1,2,3,\dots$. 
In a similar way as in the previous lemma on can show for $\ell=2,3,\dots$ that
$$
\mathbb{P}(\tau_{U_n}^\ell< L)
=\sum_{k=\ell}^{L}\mathbb{P}(\tau_{U_n}^k< L\le \tau_{U_n}^{k+1})
=\mathbb{P}(\tau_{U_n}< L)\sum_{k=\ell}^{L}\lambda_k(L,U_n)
$$
which implies as before that
$$
\lim_{L\to\infty}\lim_{n\to\infty}\frac{\mathbb{P}(\tau_{U_n}^\ell< L)}{L\mu(U_n)}=\alpha_\ell.
$$

\end{remark}

\section{The Compound Binomial Approximation} \label{binomial}

In this section we state an approximation theorem from~\cite{HV20} that provides an estimate
how closely the level sets of the counting function $W$ is approximated by a
 compound binomial distribution which represents the independent
 case. As the measure of the approximating target set $B_\rho(\Gamma)$ goes
 to zero, the compound binomial distribution then converges to a compound
 Poisson distribution.

To be more precise, the following abstract  approximation theorem which establishes the distance
between sums of $\{0,1\}$-valued dependent random variables $X_n$ and a random variable that
has a compound Binomial distribution
is used in Section~\ref{set_up_T1} in the proof of Theorem~1 to compare
 the number of occurrences in a finite time interval with the number of occurrences
 in the same interval
 for a compound binomial process.

 Let $Y_j$ be $\mathbb{N}$ valued i.i.d.\ random variables and
 denote $\lambda_\ell=\mathbb{P}(Y_j=\ell)$.
 Let $N'$ be a (large) positive integer, $t>0$ a parameter and put $p=t/N'$.
 If $Q$ is a binomially distributed random variable with parameters $(N',p)$,
 that is $\mathbb{P}(Q=k)=\binom{N'}kp^k(1-p)^{N'-k}$,
 then $W=\sum_{i=1}^QY_i$ is compound binomially distributed.
 As $N'$ goes to infinity, $Q$ converges to
 a Poisson distribution with parameter $t$ and $W$ converges to a compound
 Poisson distribution with parameters $t$, $\lambda_\ell$. (See~\cite{Gerber}.)

 Let $(X_n)_{n \in \mathbb{N}}$ be a stationary $\{0,1\}$-valued process and
 put $Z^L=\sum_{i=0}^{L-1}X_i$ for $L\in\mathbb{N}$. Let $W_a^b=Z^b-Z^a=\sum_{i=a}^{b-1} X_i$
 for $0\le a< b$ and $W=Z^N$.
 (In the following theorem we assume for simplicity's sake that $N'$ and $\Delta$ are integers.)

\begin{thrm} \label{helper.theorem}\cite{HV20}
Let $L\ll N$ and denote by $\tilde\nu$
 be the compound binomial distribution measure where the binomial part
 has values $p=\mathbb{P}(Z^L\ge1)$ and $N'=N/L$ and the compound part has
 probabilities  $\lambda_\ell=\mathbb{P}(Z^L=\ell)/p$ .
 
Then there exists a constant $C_1$, independent of $L$ and $\Delta<L$,
such that
$$
\abs{\mathbb{P}(W=k) - \tilde\nu(\{k\})}
 \leq C_1(N'(\mathcal{R}_1 + \mathcal{R}_2) + \Delta\mathbb{P}(Z^L\ge1)),
$$
where
\begin{eqnarray*}
\mathcal{R}_1 &=& \sup_{\substack {0< \Delta<M\le N'\\0 <q<N'-\Delta-1/2}}
 \left|\sum_{u=1}^{q-1}\!\left(\mathbb{P}\!\left(Z^L=u \land W_{\Delta L}^{ML}=q-u\right)
-\mathbb{P}(Z^L=u)\mathbb{P}\!\left(W_{\Delta L}^{ML}=q-u\right)\right)\right|\\
\mathcal{R}_2 &=& \sum_{j=2}^\Delta \mathbb{P}(Z^L\ge1 \land W_{jL}^{(j+1)L}\ge1).
\end{eqnarray*}
\end{thrm}

\section{Proof of Theorem~\ref{main.theorem}}\label{proof.theorem}

In this section we bound the quantities in the assumption of Theorem~\ref{helper.theorem}
in the usual way by making a distinction between short interactions, i.e.\ those that
are limited by a gap of some length $\Delta$, and long interactions which constitute the
principal part. The near independence of long interactions is expressed bv the decay of
correlations and gives rise to the error term $\mathcal{R}_1$. The short interactions
are estimated by $\mathcal{R}_2$ and use the assumptions on limited distortion,
the fact that `cylinders' are pull-backs of uniformly sized balls and
the positivity of the local dimension.

\subsection{Compound binomial approximation of the return times distribution} \label{set_up_T1}
To prove Theorem~\ref{main.theorem} we will employ the  approximation theorem from Section~\ref{binomial}
where we put $X_i=\ind_{U} \circ T^{i}$ with $U=B_\rho(\Gamma)$.
We use a standard blocking argument that was previously employed in~\cite{HV20} and
other places.
We shall use blocks of some length $L$ and
  put $Z_j=\sum_{i=jL}^{(j+1)L-1}X_i$, $j=0,1,2,\dots$. In particular $Z_j=Z_0\circ T^{jL}$ are stationary random
  variables.  
  Let us also denote $V_a^b=\sum_{j=a}^{b-1}Z_j$.
  For a positive parameter $t$ we put $N = [tL/\mathbb{P}_\mu(Z_0\ge1)]$ (non-Kac scaling)
and drop from now on the integer part brackets $[\cdot]$.
Then for any $2 \leq \Delta \leq N'=N/L=t/\mathbb{P}_\mu(Z_0\ge1)]$ (for simplicity's sake we assume
$N$ is a multiple of $L$ and $N'$ is a multiple of $\Delta$)
\begin{equation} \label{errorUSE}
\Abs{\mathbb{P}(V_0^{N'}=k) - \tilde\nu(\{k\})} \;
\leq \; C_3 ( N'(\mathcal{R}_1+\mathcal{R}_2) + \Delta \mu(Z_0\ge1)),
\end{equation}
where according to Theorem~\ref{helper.theorem}
\begin{align*}
\mathcal{R}_1 &= \sup_{\substack {0< \Delta<M\le N'\\0 <q<N'-\Delta-1/2}}
 \left|\sum_{u=1}^{q-1}\!\left(\mathbb{P}\!\left(Z_0=u \land V_{\Delta}^{M}=q-u\right)
-\mathbb{P}(Z_0=u)\mathbb{P}\!\left(V_{\Delta}^{M}=q-u\right)\right)\right|\\
\mathcal{R}_2 &= \sum_{j=1}^{\Delta} \mathbb{P}(Z_0\ge1\land Z_j\ge1),
\end{align*}
noting that $V_\Delta^M=W_{\Delta L}^{ML}$, $Z_0$ here is $Z^L$ there and 
$Z_j$ here is $W_{jL}^{(j+1)L}$ there, using the notation $W_a^b=\sum_{i=a}^{b-1} X_i$ 
from Theorem~\ref{helper.theorem}.
We also denoted by $\tilde\nu$ the compound binomial distribution with parameters
$p=\mathbb{P}(Z_j\ge1)$ and distribution $\frac{t}p\mathbb{P}(Z_j=k)$.
Notice that $\mathbb{P}(V_0^{N'}=k)=0$ for $k>N$ and also
$\tilde\nu(\{k\})=\mathbb{P}(\tilde{V}_0^{N'}=k)=0$ for $k>N$.

\vspace{0.5cm}
\noindent We now proceed to estimate the error between the distribution of $S$ and a
compound binomial  based on Theorem~\ref{helper.theorem}.

\subsection{Estimating $\mathcal{R}_1$} \label{est_R1_section}

These estimates make use of the decay of correlations and use a smoothing 
argument to approximate characteristic functions by Lipschitz continuous 
functions as was done for instance in~\cite{HV20}.

Let us fix $\rho$ for the moment and put $U=B_\rho(\Gamma)$.
Fix $q$ and $u$ and we want to estimate the quantity
$$
\mathcal{R}_1(q,u) =
 \left|\mathbb{P}\!\left(Z_0=u, V_{\Delta}^{M}=q-u\right)
-\mathbb{P}(Z_0=u)\mathbb{P}\!\left(V_{\Delta}^{M}=q-u\right)\right|
$$
In order to use the decay of correlations~(II) to obtain an estimate for
$\mathcal{R}_1(q,u)$ we approximate $\ind_{Z_0=u}$ by Lipschitz functions
from above and below as follows. Let $r>0$ be small ($r<\!\!<\rho$) and put
$U''(r)=B_{r}(U)$ for the outer approximation of $U$ and
$U'(r)=(B_{r}(U^c))^c$ for the inner approximation.
We then consider the set $\mathcal{U}=\{Z_0=u\}$ which is a disjoint
union of sets
$$
\bigcap_{j=1}^uT^{-v_j}U\cap\bigcap_{i\in[0,L)\setminus\{v_j:j\}}T^{-i}U^c
$$
where $0\le v_1<v_2<\cdots<v_u\le L-1$ the $u$ entry times vary over all
possibilities. Similarly we get its outer approximation $\mathcal{U}''(r)$ and
its inner approximation $\mathcal{U}'(r)$ by using $U''(r)$ and $U'(r)$ respectively.
 As in~\cite{HV20} we now consider Lipschitz continuous
functions approximating $\ind_{\mathcal{U}}$ as follows
\begin{equation*}
\phi_r(x) =
\begin{cases}
1 & \text{on $\mathcal{U}$} \\
0 & \text{outside $\mathcal{U}''(r)$}
\end{cases}
\hspace{0.7cm} \text{and} \hspace{0.7cm}
\hat{\phi}_r(x) =
\begin{cases}
1 & \text{on $\mathcal{U}'(r)$} \\
0 & \text{outside $\mathcal{U}$}
\end{cases}
\end{equation*}
with both linear in between. If one puts  $a=\sup_{x\in\mathcal{G}}|DT(x)|$ then
one can see that the Lipschitz norms of both $\phi_r$ and $\hat{\phi}_r$
 are bounded by $a^L/r$.
  By construction $\hat{\phi}_r \leq \ind_{Z_0=u} \leq \phi_r$.
 Also note that since 
  $$
  \mathcal{U}''(r)\setminus\mathcal{U}'(r)\subset\bigcup_{i=0}^{L-1} T^{-i}(B_{\rho+r}(\Gamma)\setminus B_{\rho-r}(\Gamma))
  $$
  one has by Assumption~(IV)
  $$
 \mu( \mathcal{U}''(r)\setminus\mathcal{U}'(r))\lesssim L\frac{r^\xi}{\rho^\beta}.
 $$

The following approximation will allow us to use the decay of correlations on the first term one
the right hand side:
\begin{align*}
\mathbb{P}\!\left(Z_0=u, V_{\Delta}^{M}=q-u\right)
-\mathbb{P}(Z_0=u)\mathbb{P}\!\left(V_{\Delta}^{M}=q-u\right)\hspace{-3cm} \\
& \leq \int_\Omega \phi_r \; \cdot \ind_{V_{\Delta}^{M}=q-u} \, d\mu - \int_\Omega \ind_{Z_0=u} \, d\mu \, \int_\Omega \ind_{V_{\Delta}^{M}=q-u} \, d\mu \\[0.2cm]
& =\mathbb{X}+\mathbb{Y}
\end{align*}
where
\begin{align*}
\mathbb{X}&=\left(\int_\Omega \phi_r \, d\mu - \int_\Omega \ind_{Z_0=u} \, d\mu \right)
 \int_\Omega \ind_{V_{\Delta}^{M}=q-u} \, d\mu\\
\mathbb{Y}&=\int_\Omega \phi_r \; (\ind_{V_{\Delta}^{M}=q-u} ) \, d\mu
- \int_\Omega \phi_r \, d\mu \, \int_\Omega \ind_{V_{\Delta}^{M}=q-u} \, d\mu .
\end{align*}
We now estimate the two terms $\mathbb{X}$ and $\mathbb{Y}$ separately.
For the first term we obtain
$$
\mathbb{X} \leq\mathbb{P}(V_{\Delta}^{M}=q-u)  \, \int_\Omega (\phi_r - \ind_{Z_0=u}) \, d\mu
 \leq \mu(\mathcal{U}''(r) \setminus \mathcal{U}(r))
 \lesssim L\frac{r^\xi}{\rho^\beta}\mu(\mathcal{U})
 =L\rho^{\xi w-\beta}\mu(\mathcal{U})
$$
if we put $r=\rho^w$ for some $w>1$.
For the second term $\mathbb{Y}$ we use the decay of correlations.
and obtain by assumption~(II) that
\begin{eqnarray*}
\mathbb{Y}&=&\left|  \int_\Omega \phi_r \; T^{-\Delta}(\ind_{V_{0}^{M-\Delta}=q-u}  ) \, d\mu
- \int_\Omega \phi_r   \, d\mu \, \int_\Omega \ind_{V_0^{M-\Delta}=q-u} \, d\mu\right|\\
&\le& \mathcal{C}(\Delta)\|\phi_r\|_{Lip}
  \end{eqnarray*}
as $\| \ind_{V_0^{M-\Delta}=q-u}\|_{\mathscr{L}^\infty}=1$.
Hence
\begin{eqnarray*}
\mu(\mathcal{U} \cap T^{-\Delta}\{V_{0}^{M-\Delta}=q-u\})
- \mu(\mathcal{U}) \, \mathbb{P}(V_{0}^{M-\Delta}=q-u)
\hspace{-4cm}&&\\
&\lesssim &a^{L}\frac{\mathcal{C}(\Delta)}{ r} + L\rho^{\xi w-\beta}\mu(\mathcal{U})
\end{eqnarray*}
As we get similar estimates using the inner approximation $\hat\phi_r$ we obtain 
$$
\mathcal{R}_1
\lesssim a^{L}\frac{\mathcal{C}(\Delta)}{ \rho^w} + L\rho^{\xi w-\beta}\mathbb{P}(Z_0\ge1)
$$
using the fact that we have the inclusion $\mathcal{U}=\{Z_0=u\}\subset\{Z_0\ge1\}$ for any $u\ge1$.

If we look at the exponential case when $\delta(n)=\mathcal{O}(\vartheta^n)$, for some $\vartheta<1$,  we then we can choose
$\Delta=\mathfrak{t}\abs{\log\rho}$ for some $\mathfrak{t}>w/\abs{\vartheta}$ and consequently get the estimate
$$
\mathcal{R}_1
\lesssim a^{L}\rho^{\mathfrak{t}\abs{\vartheta}-w}+ L\rho^{\xi w-\beta}\mathbb{P}(Z_0\ge1).
$$

\subsection{Estimating the  terms  $\mathcal{R}_2$}
To estimate the contributions made by short return times we use, as it is typically done,
expansiveness and the bounds on distortion. In order to 
 estimate the measure of $U\cap T^{-j}U$ for some positive $j$ define similarly to~\cite{HV20,CC13}
$$
\mathscr{C}_j(U)=\{\zeta_{\varphi}: \zeta_{\varphi}\cap U\not=\varnothing,\varphi\in \mathscr{I}_j\}
$$
 for the cluster of $j$-cylinders that cover the set $U$. If we put $\mathcal{V}=\{Z_0\ge1\}$,
then
\begin{eqnarray*}
\mu(T^{-j}\mathcal{V}\cap U\cap G_j)
&\le&\sum_{\zeta\in\mathscr{C}_j(U)\cap\mathcal{G}_j}\frac{\mu(T^{-j}\mathcal{V}\cap \zeta)}{\mu(\zeta)}\mu(\zeta)\\
&\lesssim&\sum_{\zeta\in \mathscr{C}_j(U)\cap\mathcal{G}_j}\mathfrak{D}(j)
\frac{\mu(\mathcal{V}\cap T^j\zeta)}{\mu(T^j\zeta)}\mu(\zeta)
\end{eqnarray*}
Since  the sets $\zeta_{\varphi}$ are $\varphi$-pre-images of $R$-balls,
the denominator is uniformly bounded from below because
$\mu(T^j\zeta)=\mu(B_{R}(y_k))$
Thus, by assumption (I), 
$$
\mu(T^{-j}\mathcal{V}\cap U\cap G_j)
\lesssim \mathfrak{D}(j)\mu(\mathcal{V}) \sum_{\zeta\in \mathscr{C}_j(U)\cap\mathcal{G}_j}\mu(\zeta)
\lesssim \mathfrak{D}(j)\mu(\mathcal{V})\,  \mathfrak{N}\, \mu\!\left(\bigcup_{\zeta\in \mathscr{C}_j(U)\cap\mathcal{G}_j}\zeta\right),
$$
where $\mathfrak{N}$ is the bound on the overlaps of cylinders $\zeta$ as given in Section~\ref{assumptions}.
Now, since $\mbox{diam }\zeta\le\delta(j)\lesssim j^{-\mathfrak{k}}$ for $\zeta\in\mathcal{G}_n$, in the
polynomial case, one has
 $$
 \bigcup_{\zeta\in \mathscr{C}_j(U)\cap\mathcal{G}_j}\zeta
\subset B_{\delta(j)}(U).
$$
Since by assumption
$\mu(B_{\delta(j)}(U))=\mathcal{O}((\delta(j)+\rho)^{d_0})$  we get in the polynomial case
$$
\mu(T^{-j}\mathcal{V}\cap U\cap G_j)
\lesssim \mathfrak{D}(j)\mu(\mathcal{V})(\delta(j)^{d_0}+\rho^{d_0})
\lesssim \mathfrak{D}(j)\mu(\mathcal{V})(j^{-\mathfrak{k} d_0}+\rho^{d_0}).
$$
On the set $G_j^c$ we use Assumption~(II-iv) and obtain
$$
\mu(T^{-j}\mathcal{V}\cap U\cap G_j^c)
\le \mu(\mathcal{V})\frac{\mu(U\cap T^{-j}\mathcal{V}\cap G_j^c)}{\mu(\mathcal{V})}
$$

For the estimate of $\mathcal{R}_2$ there are two cases to consider, namely (I) if $j\ge2$ where we have 
a gap of length $L$ to give us some dacay, and (II) when $j=1$ in which case we have to open a gap to
achieve some decay.

(I) If $j\ge2$ then we use the decomposition
$$
\{Z_0\ge1, Z_j\ge1\}
= \mathcal{V}\cap T^{-jL}\mathcal{V}
=T^{-jL}\mathcal{V}\cap \bigcup_{k=0}^{L-1}T^{-k}U
$$
 and obtain
$$
\mathbb{P}(Z_0\ge1, Z_j\ge1)
\le \sum_{k=0}^{L-1}\mu(T^{-k}U\cap T^{-jL}\mathcal{V})
=\sum_{u=(j-1)L}^{jL-1}\mu(U\cap T^{-u}\mathcal{V}).
$$
Consequently 
\begin{eqnarray*}
\sum_{j=2}^\Delta\mathbb{P}(Z_0\ge1\land Z_j\ge1)
&\le&\sum_{u=L}^{\Delta L-1}\mu(U\cap T^{-u}\mathcal{V})\\
&\lesssim&\mu(\mathcal{V})\sum_{u=L}^{\Delta L-1}
\left(\mathfrak{D}(u)(u^{-\mathfrak{k} d_0}+\rho^{d_0})+\frac{\mu(U\cap T^{-u}\mathcal{V}\cap G_u^c)}{\mu(\mathcal{V})}\right)\\
&\lesssim&\mu(\mathcal{V})\!\left(L^{-\sigma}+(L\Delta)^{1+\mathfrak{d}}\rho^{d_0}+\mathfrak{G}_L\right)
\end{eqnarray*}
since $\mathfrak{D}(u)=\mathcal{O}(u^{\mathfrak{d}})$, provided
$\sigma=\mathfrak{k} d_0-\mathfrak{d}-1$ is larger than  $0$,
where 
$$
\mathfrak{G}_L=\sum_{u=L}^\infty\frac{\mu(U\cap T^{-u}\mathcal{V}\cap G_u^c)}{\mu(\mathcal{V})}
$$
goes to zero as $L\to\infty$ by Assumption~(II-iv).

(II) If $j=1$ let $\alpha=\frac1{1+\sigma}$ and put $Z'_0=\sum_{i=L-L'+1}^{L-1}X_i$
and $Z''_0=Z_0-Z'_0$, where  $L'=L^\alpha$.  Then
$$
\mathbb{P}(Z_0\ge1, Z_1\ge1)
\le\mathbb{P}(Z''_0\ge1, Z_1\ge1)+\mathbb{P}(Z'_0\ge1),
$$
where $\mathbb{P}(Z'_0\ge1)= \mu(\mathcal{V}')$, where $\mathcal{V}'=\{Z_0'\ge1\}$. Similar to the case~(I) above we have now a gap
of length $L'$ which allows us to estimate in the same way 
\begin{eqnarray*}
\mathbb{P}(Z''_0\ge1, Z_1\ge1)
&\le&\sum_{u=L'}^{L-1}\mu(U\cap T^{-u}\mathcal{V})\\
&\lesssim & \mu(\mathcal{V})\!\left(L'^{-\sigma}+L^{1+\mathfrak{d}}\rho^{d_0}+\mathfrak{G}_{L'}\right)
\end{eqnarray*}
we conclude that
$$
\mathbb{P}(Z_0\ge1, Z_1\ge1)
\lesssim \mu(\mathcal{V})(L^{-\sigma\alpha}+L^{1+\mathfrak{d}}\rho^{d_0}+\mathfrak{G}_{L'})+\mu(\mathcal{V}').
$$

Finally if we combine steps (I) and (II) then the entire error term can be estimated by
\begin{eqnarray*}
N'\mathcal{R}_2
&\le& N'\sum_{j=1}^\Delta\mathbb{P}(Z_0\ge1, Z_j\ge1)\\
&\lesssim& N'\mu(\mathcal{V})(L^{-\alpha\sigma}+(L\Delta)^{1+\mathfrak{d}}\rho^{d_0}+\mathfrak{G}_{L^\alpha})
+N'\mu(\mathcal{V}')\\
&\lesssim& t\!\left(L^{-\alpha\sigma}+(L\Delta)^{1+\mathfrak{d}}\rho^{d_0}+\mathfrak{G}_{L^\alpha}\right)
+N'\mu(\mathcal{V}')\\
&\lesssim&L^{-\sigma\alpha}+L^{1+\mathfrak{d}}\rho^{v'}+\mathfrak{G}_{L^\alpha}+\frac{\mu(\mathcal{V}')}{\mu(\mathcal{V})}
\end{eqnarray*}
assuming $v'=d_0-v(1+\mathfrak{d})>0$ (as $\Delta=\rho^{-v}$), as $N'=t/\mu(\mathcal{V})$.

In the exponential case when $\delta(n)=\mathcal{O}(\vartheta^n)$, for some $\vartheta<1$, then we put
$\Delta=\mathfrak{t}\abs{\log\rho}$, for a suitable $\mathfrak{t}$, and obtain
\begin{eqnarray*}
\sum_{j=2}^\Delta\mathbb{P}(Z_0\ge1\land Z_j\ge1)
&\lesssim&\mu(\mathcal{V})\sum_{u=L}^{\Delta L-1}
\left(u^{\mathfrak{d}}(\vartheta^{u d_0}+\rho^{d_0})+\frac{\mu(U\cap T^{-u}\mathcal{V}\cap G_u^c)}{\mu(\mathcal{V})}\right)\\
&\lesssim&\mu(\mathcal{V})\!\left(\tilde\vartheta^L+L^{1+\mathfrak{d}}\rho^{\tilde{d}_0}+\mathfrak{G}_L\right)
\end{eqnarray*}
for any $\tilde\vartheta\in(\vartheta,1)$ and any positive $\tilde{d}_0<d_0$. This yields
$$
N'\mathcal{R}_2
\lesssim \tilde\vartheta^{L^\alpha}+L^{1+\mathfrak{d}}\rho^{\tilde{d}_0}+\mathfrak{G}_{L^\alpha}+\frac{\mu(\mathcal{V}')}{\mu(\mathcal{V})}.
$$

\subsection{The total error}\label{total_error} For the total error  we now put $r=\rho^w$
and as above $\Delta=\rho^{-v}$ where $v<d_0$ since $\Delta\ll N$ and $N\ge \rho^{-d_0}$.
Moreover $L'=L^\alpha$ for $\alpha=1/(1+\sigma)$
and in the polynomial case when $\mathcal{C}(\Delta)=\mathcal{O}(\Delta^{-p})=\mathcal{O}(\rho^{pv})$
we get
\begin{eqnarray*}
\left|\mathbb{P}(W=k)-\tilde\nu(\{k\})\right|\hspace{-3cm}&&\\
&\lesssim& N'\!\left(a^{L}\frac{\mathcal{C}(\Delta)}{ \rho^w} + L\rho^{\xi w-\beta}\mu(\mathcal{U})\right)
+L^{-\alpha\sigma}+L^{1+\mathfrak{d}}\rho^{v'}+\mathfrak{G}_{L^\alpha}+\frac{\mu(\mathcal{V}')}{\mu(\mathcal{V})}
+\Delta\mu(\tau_U\le L)\\
&\lesssim&\frac{a^{L}}L \rho^{v\mathfrak{p}-w-d_1}+\rho^{w\xi-\beta}+L^{1+\mathfrak{d}}\rho^{v'}+L^{-\alpha\sigma}+\mathfrak{G}_{L^\alpha}
+\frac{\mu(\mathcal{V}')}{\mu(\mathcal{V})}+L\rho^{d_0-v}
\end{eqnarray*}
as $N'\mu(U)=\frac{s}{L}$, $s=N'\mathbb{P}(Z_0\ge1)$, $N=t/\mathbb{P}(Z_0\ge1)\gtrsim \rho^{d_1}$
and $\Delta\mu_(\tau_U\le L)\lesssim \rho^{-v+}\rho^{d_0}$.
When $\rho\to0$ then $\mu(U)\to 0$ and in order to get convergence we require $v\mathfrak{p}-w-d_1>0$,
$w\xi-\beta>0$ and $v'=d_0-v(1+\mathfrak{d})>0$. This can be achieved if $w>\beta/\xi$ is sufficiently close to $\beta/\xi$ and
$\mathfrak{p}>\left(\frac\beta\xi +d_1\right)\frac{1+\mathfrak{d}}{d_0}$
 in the case when $\mathcal{C}$ decays polynomially
 with power $\mathfrak{p}$, i.e.\ $\mathcal{C}(k)\sim k^{-\mathfrak{p}}$.
 These choices also satisfy $v<d_0$.
 Since we also must have $\sigma=\mathfrak{k}d_0-\mathfrak{d}-1>0$ this is satisfied if 
 $\mathfrak{k}>\frac{\mathfrak{d}+1}{d_0}$.

 In the exponential case ($\diam\zeta=\mathcal{O}(\vartheta^n)$ for $n$ cylinders $\zeta$ and
 $\mathcal{C}(\Delta)\sim\vartheta^\Delta$)
we obtain with $\Delta=s|\log \rho|$ for $s$ large enough
\begin{eqnarray*}
|\mathbb{P}(W=k)-\tilde\nu(\{k\})|\hspace{-3cm}&&\\
&\lesssim&
a^{L}\rho^{s\abs{\log\vartheta}-w-d_1}+\rho^{ws-\beta}+L^{1+\mathfrak{d}}\abs{\log\rho}^{1+\mathfrak{d}}
+\tilde\vartheta^{L^\alpha}+\frac{\mu(\mathcal{V}')}{\mu(\mathcal{V})}+\mathfrak{G}_{L^\alpha}+\Delta\mu(\tau_U< L)
\end{eqnarray*}
(the last term comes from~\eqref{errorUSE}),
for some  $\tilde\vartheta\in(\vartheta,1)$. Note that in this case we only need to have $d_0>0$.

 \subsection{Convergence to the compound Poisson distribution}
 For $t>0$ and blocklength $L\in\mathbb{N}$ we let $N'=t/\mathbb{P}(Z^L\ge1)$.
 Then we do the double limit of first letting $\rho$ go to zero and then we letting $L$ go to infinity.
 Call by $\tilde\nu_{L,\rho}$ the compound binomial distribution with the parameters
 $p=\mathbb{P}(Z^L\ge1)$ and $N'=t/p$ and compounding parameters
 $\lambda_\ell(L,\rho)=\frac1p\mathbb{P}(Z^L=\ell)$, $\ell=1,2,\dots$. Then, as we let $\rho\to0$
 this implies that $p\to0$ and therefore
the compound binomial distribution $\tilde\nu_{L,\rho}$ converges to the
compound  Poisson distribution $\tilde\nu_L$ for the parameters 
$t\lambda_\ell(L)=t\lim_{\rho\to0}\lambda_\ell(L,\rho)\forall \ell\in\mathbb{N}$.
Thus for every $L$:
 $$
 \mathbb{P}(W=k)\longrightarrow\tilde\nu_L(\{k\})+\mathcal{O}(L^{-\alpha\sigma}+\mathfrak{G}_{L^\alpha}).
$$
Now as $L\to\infty$ then $\lambda_\ell(L)\to\lambda_\ell$ for all $\ell=1,2,\dots$
and the compound Poisson distributions $\tilde\nu_L$ converges to the compound Poisson distribution $\nu$ for the
parameters $\lambda_\ell=\lim_{L\to\infty}\lambda_\ell(L)$. Now we finally obtain as in~\cite{HV20}
 $$
 \mathbb{P}(W=k)\longrightarrow\nu(\{k\})
$$
as $\rho\to0$. This concludes the proof of Theorem~\ref{main.theorem}.
\qed

\section{Examples}\label{examples}

\subsection{ $C^2$ interval maps}\label{interval.maps}

Let $T:I\to I$ is piecewise uniformly expanding
on the interval $I$. We assume that $T$ is piecewise $C^2$ 
with uniformly bounded $C^2$-norms. For simplicity sake we  assume that every interval 
of continuity gets mapped to the entire interval. If we denote by $\mathscr{I}^n$ the inverse branches of $T^n$ 
then the $n$-cylinders are 
$\zeta_\varphi=\varphi(I)$ where $\varphi\in\mathscr{I}_n$. 
If, as before,
$$
\delta(n)=\max_{\varphi\in\mathscr{I}_n}|\zeta_\varphi|
$$
then $\delta(n)$ decays exponentially fast since we assume that $|DT(x)|\ge c_0$ 
uniformly in $x$ for a constant $c_0>1$. That is $\delta(n)\le c_0^{-n}$ (his corresponds to the case
when $\mathfrak{k}=\infty$).

Then there exists a unique absolutely continuous $T$-invariant probability measure $\mu$ with
density $h$. Since $T$ is uniformly expanding the  decay of the correlation function~(I) is exponential, that is 
the decay function $\mathcal{C}(n)$ goes exponentially fast to zero (this is equivalent 
to $\mathfrak{p}=\infty$).

Since the density $h$ is bounded and bounded away from $0$, 
Condition~(V) is satisfied with any values $d_0<1<d_1$ 
arbitrarily close to $1$. 
Condition~(III) follows from the uniform boundedness of second order derivatives which implies that
$\mathfrak{d}=0$. Let us notice that $\mathcal{G}_j^c=\varnothing$ for all $j$ and that the contraction
rate of the maximal size of $n$-cylinders $\delta(n)$ is exponential.
The Annulus Condition~(VI) is satisfied with $\xi=\beta=1$.

By~\cite{HV20}, Lemma~4, $x$ is a periodic point if and only if the 
$\eta_\rho(x)=\inf\{j\ge1: T^jB_\rho(x)\cap B_\rho(x)\not=\varnothing\}$.
Note that $\eta_\rho(x)\le\eta_{\rho'}(x)$ if $\rho>\rho'>0$ and if $\rho\to0$ then
$\eta_\rho(x)$ converges to the minimal period $p$ of $x$ if $x$ is a periodic point
and to $\infty$ if $x$ is not periodic. If $x$ is periodic with minimal period $p$,
then let us put $\Gamma=\{x\}$ and we get that $\lambda_\ell =(1-\vartheta)\vartheta^{\ell-1}$, where
$$
\vartheta=\lim_{\rho\to0}\frac{\mu(B_\rho(x)\cap T^{-m}B_\rho(x))}{\mu(B_\rho(x))}=|DT^p(x)|^{-1}
$$
is the Pitskel value of $x$. If $x$ is not periodic then we find that $\lambda_1=1$ and
$\lambda_\ell=0$ for all $\ell\ge2$ which follows from the fact that $\eta_\rho(x)$ then
diverges to $\infty$.

We can therefore invoke Theorem~\ref{main.theorem} 
and obtain the following result (cf.\ e.g.~\cite{AfBun10, BunYur11}):

\begin{thrm} Let the map $T$ be piecewise $C^2$ with uniformly bounded
$C^2$ derivative and uniformly expanding.
 Then
$$
W_{x,\rho}=\sum_{j=0}^{N-1}\ind_{B_\rho(x)}\circ T^j
$$
 with respect to the measure $\mu$ is in the limit $\rho\to0$:\\
 (i) Poisson(t) if $x$ is non-periodic,\\
 (ii) P\'olya-Aeppli if $x$ is periodic with minimal period $p$ and the Pitskel value
 $\vartheta=|DT^p(x)|^{-1}$, $\lambda_\ell=(1-\vartheta)\vartheta^{\ell-1}$,\\
where $N=N_\rho(x)=\frac{t}{\mu(B_\rho(x))}$ (Kac scaling).
\end{thrm}

We can now also easily look at the case when $\Gamma$ is the union of finitely many
periodic points $x_1, x_2,\dots, x_n$ which don't have a common orbit. Let us denote by 
$p_1,p_2,\dots, p_n$ their respective minimal periods. If $L>0$ is a large number then 
for all $\rho>0$ small enough we have that $T^{-k}B_\rho(x_i)\cap B_\rho(x_j)=\varnothing$
for all $i\not=j$ and $k<L$. Put $\Gamma=\{x_i: i=1,\dots,n\}$.
Let us now note that 
$$
\{\tau^k_{B_\rho(\Gamma)}<L\}\cap B_\rho(\Gamma)
=\bigcup_{i=1}^n\bigcap_{j=0}^{k} T^{-jp_i}B_\rho(x_i)
$$
where the union on the RHS is disjoint for all $\rho>0$ small enough.
This leads to
$$
\mathbb{P}(\{\tau^k_{B_\rho(\Gamma)}<L\}\cap B_\rho(\Gamma))
=\sum_{i=1}^n\mu\!\left(\bigcap_{j=0}^k T^{-jp_i}B_\rho(x_i)\right)
=\sum_{i=1}^n(1+o(1))h(x_i)|DT^{p_i}(x_i)|^{-k}2\rho
$$
and also
$$
\mu(B_\rho(\Gamma))
=\sum_{i=1}^n\mu(B_\rho(x_i))
=(1+o(1))2\rho \sum_{i=1}^n h(x_i).
$$
If we put $\vartheta_i=|DT^{p_i}(x_i)|^{-1}$ then 
$$
\hat\alpha_{k+1}(L,B_\rho(\Gamma))
=\mathbb{P}(\tau^k_{B_\rho(\Gamma)}<L|B_\rho(\Gamma))
=(1+o(1))\frac{\sum_{i=1}^nh(x_i)\vartheta_i^k}{\sum_{i=1}^n h(x_i)}.
$$
Consequently, as $\rho\to0$ and $L\to\infty$,
$$
\hat\alpha_{k+1}=\sum_{i=1}^n\vartheta_i^kH_i,
$$
where
$$
H_i=\frac{h(x_i)}{\sum_{j=1}^n h(x_j)}.
$$
This then gives
$$
\alpha_k=\hat\alpha_k-\hat\alpha_{k+1}
=\sum_{i=1}(1-\vartheta_i)\vartheta_i^{k-1}H_i
$$ 
and in particular the extremal index $\alpha_1=\sum_{i=1}^n(1-\vartheta_i)H_i$.
Finally we get the cluster probabilities
$$
\lambda_k=\frac{\alpha_k-\alpha_{k-1}}{\alpha_1}
=\frac{\sum_{i=1}^n(1-\vartheta_i)^2\vartheta_i^{k-1}H_i}{\sum_{i=1}^n(1-\vartheta_i)H_i}.
$$
In the case of $n=1$ we of course recover the previous formula for a single periodic point
and see that in general the limiting return times distribution is not P\'olya-Aeppli distributed.
This proves the following result:

\begin{cor} Let $T: [0,1]\to [0,1]$ be piecewise $C^2$ with uniformly bounded
$C^2$ derivative and uniformly expanding. Let $\Gamma=\{x_1,x_2,\dots,x_n\}$ be a set of 
finitely many periodic points with distinct orbits and minimal periods $p_1,\dots,p_n$.

 If $\mu$ is the absolutely continuous invariant probability measure with density $h$, then
$$
W_{\Gamma,\rho}=\sum_{j=0}^{N-1}\ind_{B_\rho(\Gamma)}\circ T^j,
$$
with $N=N_\rho(\Gamma)=\frac{t}{\mu(B_\rho(\Gamma))}$, is in the limit $\rho\to0$ compound Poisson distributed with parameters $t>0$ and 
$$
\lambda_k
=\frac{\sum_{i=1}^n(1-\vartheta_i)^2\vartheta_i^{k-1}h(x_i)}{\sum_{i=1}^n(1-\vartheta_i)h(x_i)}, \hspace{1cm} k=1,2,\dots,
$$
where $\vartheta_i=|DT^{p_i}(x_i)|^{-1}$ is the Piskel value at $x_i$.
\end{cor}

\subsection{Product of interval maps}\label{product.interval.maps}

Let $T_1, T_2: [0,1]\to[0,1]$ be two interval maps as in section~\ref{interval.maps}
and let us put $T=T_1\times T_2:\mathbb{T}^2\to\mathbb{T}^2$ by $T(x,y)=(T_1x,T_2y)$
for the product map. If $\mu_1, \mu_2$ are the two absolutely continuous invariant probability
measures for $T_1,T_2$ then the product measure $\mu=\mu_1\times\mu_2$ the 
absolutely $T$-invariant probability measure on $\mathbb{T}^2$. 
Clearly correlations are decaying exponentially fast, i.e.\ $\mathcal{C}(k)$ of condition~(I)
decays exponentially. In condition~(II)  $\mathcal{G}_n^c=\varnothing$, $\mathfrak{d}=0$
and $\delta(n)$ decays exponentially.

Now let $x_1\in[0,1]$ be a periodic point of $T_1$ with minimal period $p$ and
Pitskel value $\vartheta=|DT^p(x)|^{-1}$. Let us take $\Gamma=\{x\}\times[a,b]\subset\mathbb{T}^2$
for some $0\le a < b \le 1$. Then $\mu(\Gamma)=0$. Let us note that condition~(III) is
satisfied with $d_0=d_1=2$ and~(IV) is satisfied with $\xi=\beta=1$.

To determine the clustering parameters note that for any $L>0$ one has 
 $B_\rho(\Gamma)\cap T^{-\ell}B_\rho(\Gamma)=\varnothing$ for all $\rho>0$ small enough and 
 all $\ell\not\in p\mathbb{N}$ that satisfy $\ell<pL$.  As before $\tau^k_B $ denotes
 the $k$th return/entry time for the map $T$ to the set $B\subset\mathbb{T}^2$.
 In addition let $\hat\tau_U^k$ be the $k$th return/entry time
of the interval map $T_2^p$ to the set $U\subset[0,1]$, which for instance in the case  $k=1$ means
$\hat\tau_U^1(y)=\inf\{j\ge 1: T_2^{jp}y\in U\}$.
If we put 
 $$
 \gamma_k(i)=\mathbb{P}_{\mu_2}(\hat\tau_{[a,b]}^k=i|[a,b])
 $$
 for the conditional size of the levelsets of $\hat\tau_{[a,b]}^k$, then
 $$
 \mathbb{P}_\mu(\tau_{B_\rho(\Gamma)}<L|B_\rho(\Gamma))
 =(1+o(1))\sum_{i=k}^{L-1}|DT^{pi}(x)|^{-1}\mathbb{P}_{\mu_2}(\hat\tau_{[a,b]}^k=i|[a,b])
 \longrightarrow \sum_{i=k}^{L-1}\vartheta^i\gamma_k(i)
 $$
 in the limit as $\rho\to0$.  As $L\to\infty$ this yields
 $$
 \hat\alpha_{k+1}= \sum_{i=k}^{\infty}\vartheta^i\gamma_k(i)
 $$
 which leads to $\alpha_k= \sum_{i=k}^{\infty}\vartheta^i(\gamma_k(i)-\gamma_{k+1}(i))$
 with extremal index  $\alpha_1= \sum_{i=k}^{\infty}\vartheta^i(\gamma_1(i)-\gamma_{2}(i))$
 and cluster probabilities 
 $\lambda_k=\frac1{\alpha_1}\sum_{i=k}^{\infty}\vartheta^i(\gamma_k(i)-2\gamma_{k+1}(i)+\gamma_{k+2}(i))$.
 This proves the following result:

\begin{thrm} Let $T_1, T_2: [0,1]\to [0,1]$ be uniformly expanding piecewise $C^2$ with uniformly bounded
$C^2$ derivative and $T=T_1\times T_2:\mathbb{T}^2\to\mathbb{T}^2$ the
product map. Let $\mu=\mu_1\times\mu_2$ the $T$-invariant absolutely continuous $T$-invariant probability measure.
 Let $x\in[0,1]$ be a periodic point of $T_1$ with minimal period $p$ and  Pitskel value $\vartheta=|DT^p(x)|^{-1}$.
  Moreover let $0\le a< b\le1$ and put 
  $$
 \gamma_k(i)=\mathbb{P}_{\mu_2}(\hat\tau_{[a,b]}^k=i|[a,b]).
 $$

 If $\Gamma=\{x\}\times[a,b]$, then
$$
W_{\Gamma,\rho}=\sum_{j=0}^{N-1}\ind_{B_\rho(\Gamma)}\circ T^j,\hspace{6mm}
N=N_\rho(\Gamma)=\frac{t}{\mu(B_\rho(\Gamma))},
$$
is in the limit $\rho\to0$ compound Poisson distributed with parameters $t>0$ and 
$$
 \lambda_k=\frac1{\alpha_1}\sum_{i=k}^{\infty}\vartheta^i(\gamma_k(i)-2\gamma_{k+1}(i)+\gamma_{k+2}(i))
 \hspace{1cm} k=1,2,\dots,
$$
where   $\alpha_1= \sum_{i=k}^{\infty}\vartheta^i(\gamma_1(i)-\gamma_{2}(i))$.
\end{thrm}

As a special case we can consider the map $T_2(y)=2y\mod 1$, the doubling map,
and the interval $[a,b]=[0,\frac12]$. Then 
$[0,\frac12]\cap T_2^{-p}[0,\frac12]=\bigcup_{u=0}^{2^p\frac12-1}2^{-p}(u+[0,\frac12])$,
where $u+[0,\frac12]$ denotes the interval $[u,u+\frac12]$.
By iteration one obtains the disjoint union representation
$$
\bigcap_{i=0}^k T^{-ip}\left[0,\frac12\right]
=\bigcup_{0\le u_i<2^p\frac12-1, i=1,\dots,k}\left( \sum_{i=1}^k \frac{u_i}{2^{ip}}\right)+\frac1{2^{kp}}\left[0,\frac12\right]
$$
which contains $(\frac122^p)^k$ many disjoint intervals of length $\frac122^{-kp}$. Thus
$\mu_2\!\left(\bigcap_{i=0}^k T^{-ip}\left[0,\frac12\right]\right)=(\frac12)^{k+1}$, as $\mu_2$
is the Lebesgue measure, and therefore
$\gamma_k(k)=2^{-k}$.
In this simple case the limiting distribution of $W_{\Gamma,\rho}$ is P\'olya-Aeppli for
$\frac\vartheta2$ with the extremal index $\alpha_1=1-\frac\vartheta2$, where
$\Gamma=\{x\}\times[0,\frac12]$ and $x$ is as above a periodic point of $T_1$ with minimal
period $p$ and Pitskel value $\vartheta=|DT_1^p(x)|^{-1}$.

\subsection{Parabolic interval maps}

Let us consider the Pomeau-Manneville map which for a parameter $\alpha\in(0,1)$ is given
by 
$$
T(x)=\begin{cases}x+2^{\alpha}x^{1+\alpha}&\mbox{ if $x\in[0,\frac12)$}\\
2x-1&\mbox{ if $x\in[\frac12,1]$}\end{cases}
$$
and has a parabolic point at $x=0$, i.e.\ $T'(0)=1$. Otherwise $T'(x)>1$ for all $x\not=0$.
These maps have a neutral (parabolic) fixed point at $x=0$ and are otherwise
expanding. It is known that $T$ has an invariant absolutely continuous 
probability measure with density $h(x)$, where $h(x)\sim x^{-\alpha}$ for $x$ close to $0$.

 There exists a constant $C$ so that
$$
\left| \int\psi(\phi\circ T^n)\,d\mu
 -\int\psi\,d\mu\int\phi\,d\mu\right|
\le  C \|\psi\|_{Lip}\|\phi\|_\infty\frac1{n^\gamma},
$$
where $\gamma=\frac1{\alpha}-1$.

  Consequently, for the purposes of Theorem~\ref{main.theorem}, Assumption~(I) is satisfied with 
 $\lambda(n)=\mathcal{O}(n^{-\mathfrak{p}})$ with $\mathfrak{p}=\frac1{\alpha}-1$.
 Clearly the dimensions of $\mu$ and $\mu^\omega$ are equal to one and Assumption~(V)
 is satisfied with any $d_0<1<d_1$ arbitrarily close to $1$ for any point $x$ away from the 
 parabolic point at $0$.  Assumption~(VI) is satisfied with $\xi=\beta=1$.

 Let us denote by $\psi_{0}$ for the parabolic inverse branch of $T$ and 
 denote by $\psi_{1}$ the other inverse branch $\psi_1(x)=\frac12+\frac{x}2$. 
 Then $\psi_0^n$ the (unique) inverse branch of  $T^n$ 
 which contains the parabolic point $0$, then one has that 
 $|\psi_0^n(I)|=\mathcal{O}(n^{-1/\alpha})$, where $\psi_0^n(I)=[0,a_n]$, $a_n=\psi_0^n(1)\sim n^{-\frac1\alpha}$.

One has
 $(\psi_{0}^k)'(1)=1/(T^k)'(a_{k})$ and since
 $T'(s)=1+(1+\alpha)2^\alpha x^\alpha$ we obtain
 \begin{eqnarray*}
 (T^k)'(a_{k})&=&\prod_{\ell=1}^k\left(1+(1+\alpha)2^\alpha a_{\ell}^\alpha\right)\\
 &\sim&\exp\sum_{\ell=1}^k(1+\alpha)2^\alpha a^\alpha\frac1\ell\\
 &\sim&\exp\left((1+\alpha)(2\alpha)^\alpha \log k\right)\\
 &=&k^c,
 \end{eqnarray*}
 where $c=\frac{1+\alpha}{(2\alpha)^\alpha}=\frac{1+\alpha}\alpha=1+\frac1\alpha$.

Now for $j\in\mathbb{N}$,  let us put $\zeta=\psi_{\vec{i}}(I)$ for the $j$-cylinder given 
by $\vec{i}=(i_1,i_2,\dots,i_j)\in\{0,1\}^j$, where 
$\psi_{\vec{i}}=\psi_{i_j}\circ\cdots \psi_{i_2}\circ \circ\psi_{i_1}$.
 The $\zeta$ then  describe all possible $j$-cylinders as $\vec{i}$ ranges over all possibilities.
 For some $\beta<1$ to be determined later, let us put $q=j^\beta$ and define the set of 'bad'
 $j$-cylinders $\mathcal{G}_{j}^c$ by
 $$
 \mathcal{G}_{j}^c
 =\left\{\zeta=\psi_{\vec{i}}(I): \vec{i}\in\{0,1\}^j, i_1=i_2=\cdots=i_q=0\right\}
 $$
 and put $G_{j}^c=\bigcup_{\zeta\in\mathcal{G}_{j}^c}\zeta$.
 Then
 $$
 \mu(G_{j}^c)\le \mu(U_q)
 \lesssim \frac1{q^\gamma}.
 $$
 Now let $n_0$ be so that $x\in U_{n_0}\setminus U_{n_0-1}$,
 then $\zeta\cap B_\rho(x)\not=\varnothing$ if $\zeta=\psi_{\vec{i}}(I)$,
 for $i_{j-n_0}=i_{j-n_0+1}=\cdots=i_j=0$.

 \subsubsection{Away from the parabolic point}
  From now on we let $\Gamma=\{x\}$ for $x\in I$, $x\not=0$.
 In order to get better estimates of the term $\mathcal{R}^2$ we will, unlike done earlier, stratify the set
  $\mathcal{G}_j^n$ as follows.
 If $\zeta\in\mathcal{G}_{j}$ then there exists $s<q$ so that $\zeta=\psi_{\vec{i}}(I)$,
 where $i_1=i_2=\cdots=i_s=0$ and $i_{s+1}=1$. In this case, if moreover $\zeta\cap B_\rho(x)\not=\varnothing$,
 then 
 $$
 \mbox{diam}(\zeta)
 \lesssim\mbox{diam}(U_{s})\frac1{(j-s-n_0)^{c}}\frac1{n_0^{c}}.
 $$
Or, since 
 $\mbox{diam}(U_s)\le a_{s}\lesssim s^{-\frac1{\alpha}}$ we get
 $$
 \delta_s(j)\lesssim \frac1{s^{\frac1{\alpha}}}\frac1{((j-s-n_0)n_0)^{c}}
 $$
 which is the diameter estimate for those $j$-cylinders for which have the given value for $s$.
 In a similar way we can estimate the distortion of those cylinders by
 $$
 \mathfrak{D}_s(j)\lesssim\mbox{distortion}\!\left(T_{\alpha}^s|_{U_s}\right)
 \lesssim s^{c}.
 $$

 Since the dimension of the measures $\mu$ on the unit interval $I$ is equal to $1$, we 
 get the following refined estimate of the error term $\mathcal{R}_2$ where the gap is again denoted
 by $\Delta$. For the contribution of the short returns  we get (as $q<\!\!<j$)
 \begin{eqnarray*}
 \mathcal{R}_{2,\mbox{\tiny good}}&\lesssim&\sum_{j=L'}^\Delta\sum_{s=0}^{q-1}\delta_s(j)\mathfrak{D}_s(j)\\
 &\lesssim&\sum_{j=L'}^\Delta\sum_{s=0}^{q-1}\frac{s^{1+\frac1\alpha}}{s^{\frac1{\alpha}}((j-s-n_0)n_0)^{c}}\\
 &\lesssim&\sum_{j=L'}^\Delta\frac{q^{2}}{j^{c}}\\
 &\lesssim&\frac1{L'^{\frac1{\alpha}-2\beta}}
 \end{eqnarray*}
 provided $\frac1{\alpha}-2\beta$ is positive, i.e.\ $\beta<\frac1{2\alpha}$ and $L'<L$.
 Therefore
 $$
 \mathcal{R}_2
 \lesssim \frac1{L'^{\frac1{\alpha}-2\beta}}+\mathfrak{G}_{L'}.
 $$

 The estimates of the other terms $\mathcal{R}_1$ and $\mathcal{R}_3$ proceed unchanged.
 As before we approximate the characteristic function $\ind_{B_\rho(x)}$ from the outside by a Lipschitz continuous
 function $\phi$ which smoothens on an annulus of thickness $\rho^w$ for a $w>1$ and in a similar
 way with a Lipschitz function $\hat\phi$ from the inside. Then 
 $\|\phi\|_{\mbox{\tiny Lip}}, \|\hat\phi\|_{\mbox{\tiny Lip}}\lesssim \rho^{-w}$ and gives  for the correlation term:
 $$
 \mathcal{R}_{1,\mbox{\tiny corr}}\lesssim a^L\rho^{-w}\frac1{\Delta^{\frac1{\alpha}-1}}
 \lesssim\rho^{v(\frac1{\alpha}-1)-w},
 $$
 where we used the decay of correlation and where we put, 
 as before $\Delta=\rho^{-v}$ for a $v<1$. This tells us that the condition on the parameters
 is $v(\frac1{\alpha}-1)-1-w>0$.
 For the annulus term we get as before
 $$
 \mathcal{R}_{1,\mbox{\tiny ann}}\lesssim \rho^w
 \lesssim \rho^{w}\mathbb{P}(Z^L\ge1)
 $$
 which goes to zero  as $w$ is larger than $1$. Then 
 $\mathcal{R}_1=\mathcal{R}_{1,\mbox{\tiny corr}}+\mathcal{R}_{1,\mbox{\tiny ann}}$.
 The third error term is then 
 $$
 \mathcal{R}_3
 \lesssim\sum_{k=0}^N\sum_{j=1}^\Delta \mu(\mathcal{V})^2
 =N\Delta\mu(\mathcal{V})^2
 \lesssim \Delta\mu(\mathcal{V})
 $$
 which goes to zero as $v<1$. 
 The total error now is
 $$
 \mathcal{R}\lesssim N'(\mathcal{R}_1+\mathcal{R}_2)+\mathcal{R}_3.
 $$
 
 Since   $\mathfrak{g}=\beta\gamma$ which implies that $\mu(G_{j}^c)\le j^{-\beta\gamma}$
 and define $N_{L}=\sum_{j=L}^\infty\ind_{G_{j}^c}$ for which we get
 $$
 \mu(N_{L})=\sum_{j=L}^\infty\mu(G_{j}^c)
 \lesssim\sum_{j=L}^\infty\frac1{j^{\beta\gamma}}
 \lesssim\frac1{L^{\beta\gamma-1}}.
 $$
 The Hardy-Littlewood maximal function is then
 $$
 HN_{L}(x)=\sup_{\rho>0}\frac1{\mu(B_\rho(x))}\int_{B_\rho(x)}N_{L}\,d\mu
 =\sup_{\rho>0}\sum_{j=L}^\infty\frac{\mu(G_{j}^c\cap B_\rho(x))}{\mu(B_\rho(x))}
 $$ 
 for which we get by the maximal inequality for $\varepsilon>0$:
 $$
 \mathbb{P}_{\mu}(HN_{L}>\varepsilon)
 \lesssim\frac1\varepsilon\int_I N_{L}\,d\mu
 \lesssim\frac1{\varepsilon L^{\beta\gamma-1}} \hspace{1cm} \forall \rho>0.
 $$ 
 If we put $\varepsilon=L^{-(\beta\gamma-1)/2}$ then
 $$
 \mathbb{P}_{\mu}\!\left(\sum_{j=L}^\infty \frac{\mu(G_{j}^c\cap B_\rho(x))}{\mu(B_\rho(x))}
 >\frac1{L^{(\beta\gamma-1)/2}}\right)
 \lesssim\frac1{L^{(\beta\gamma-1)/2}}.
 $$
 Since $N_{L}\ge1$ if $N_{L}\not=0$ we get that 
 $\mathbb{P}_{\mu}(N_{L}\not=0)\lesssim L^{-(\beta\gamma-1)}$ decays to zero as $L\to\infty$
 if $\gamma>1$.
By~\cite{Fol} Theorem~3.18 we now obtain 
 $$
 \mathfrak{G}_{L,\rho}(x)=\sum_{j=L}^\infty \frac{\mu(G_{j}^c\cap B_\rho(x))}{\mu(B_\rho(x))}
 \longrightarrow N_{L}=0
 $$
 for all $x\in\{N_{L}=0\}$ as $\rho\to 0$. Since $\mathbb{P}(N_L\not=0)$ is summable if 
 $\beta(\frac1\alpha-1)-1>1$ and $\beta<1$ can be arbitrarily close to $1$, we see that if $\alpha<\frac13$ 
  by the Borel-Cantelli theorem 
 $$
\lim_{L\to\infty}\lim_{\rho\to0} \mathfrak{G}_{L,\rho}(x)= 0
 $$
for almost every $x \in I$, i.e.\ $\mu$-almost every $x$ lies in $\liminf_{L,\rho}\{N_L=0\}$.

Let us note that if $x\in\liminf_{L,\rho}\{N_L=0\}$ is periodic with minimal period $m$ then
we get that the limiting distribution is P\'olya-Aeppli with parameter $\vartheta=|DT^m(x)|^{-1}$
and the Kac. scaling $N=\frac{t}{\mathbb{P}(Z^L\ge1)/L}\sim\frac{\vartheta t}{\mu(B_\rho)}$.
If $x\in\liminf_{L,\rho}\{N_L=0\}$ is non-periodic then the limiting distribution is 
Poisson$(t)$ with the scaling $N=\frac{t}{\mathbb{P}(Z^L\ge1)/L}\sim\frac{ t}{\mu(B_\rho)}$.

\subsubsection{Asymptotics at the parabolic point}
If $x=0$ we obtain a different scaling since in this case $\alpha_1=0$. Instead of $\rho\to0$ we shall use the 
neighbourhoods $U_n=[0,a_n]$. Denote by $A_n=U_n\setminus U_{n+1}=(a_{n+1},a_n]$
and $V_{n,L}=\bigcup_{j=n}^{n+L-1}A_j=U_n\setminus U_{n+L}$ for $L\in\mathbb{N}$.
For given $n$ let as before $Z^N=\sum_{j=0}^{N-1}\ind_{U_n}\circ T^j$ be the hit counting 
function where we assume for simplicity's sake that $N=rL$ for some $r, L\in\mathbb{N}$. 
If $k<K$ (for some $K<L$) we obtain
\begin{eqnarray*}
\mathbb{P}(Z^N=k)
&=&\mathbb{P}\!\left(\bigcap_{\vec{k}\in\mathscr{K}_r(k)} \bigcap_{j=1}^rT^{-(j-1)L}\{Z^L=k_j\}\right)\\
&=&\mathbb{P}\!\left(\bigcap_{\vec{k}\in\mathscr{K}_r(k)} \bigcap_{j=1}^r T^{-(j-1)L}\{\hat{Z}^L=k_j\}\right)\\
&=&\mathbb{P}\!\left(\sum_{j=0}^{N-1}\hat{Z}^L\circ T^{jL}=k\right),
\end{eqnarray*}
where $\mathscr{K}_r(k)=\{\vec{k}\in\mathbb{N}_0^r:\sum_{j=1}^rk_j=k\}$
and $\hat{Z}^L=\sum_{j=1}^{L-1}\ind_{V_{n,K}}\circ T^j$. This is because if 
$Z_{U_n}^N=k<K$ then for the iterates one has $T^jx\not\in U_{n+K}$ for $j=0,\dots, N$.
Let $K=L^\eta$ for some $\eta\in(0,1)$. Then $K\to\infty$ as $L\to \infty$.

\begin{lem}
For all $\alpha\in(0,1)$:
$$
 \frac{\mathbb{P}(\hat{Z}^L\ge1)}{\mu(V_{n,L+K-1})}\longrightarrow 1
 $$
 as $n\to \infty$.
In particular there is a constant $C$ so that
$$
C^{-1}n^{-\frac1\alpha}\le \mathbb{P}(\hat{Z}^L\ge1)\le C n^{-\frac1\alpha}.
 $$
 \end{lem}
 
 \begin{proof}
 We do the following decomposition as $\psi_0^iV_{n,K}=V_{n+i,K}$:
 \begin{eqnarray*}
 \{\hat{Z}^L\ge1\}
 &=&\bigcup_{\ell=0}^{L-1}T^{-\ell}V_{n,K}\\
 &=&\bigcup_{\ell=0}^{L-1}\left(V_{n+\ell,K}\cup\bigcup_{i=0}^{\ell-1}T^{-(\ell-i-1)}\psi_1V_{n+i,K}\right)\\
  &=& V_{n,K+L-1}\cup\mathcal{E}
 \end{eqnarray*}
 where we get the error term ($k=\ell-i$)
 $$
 \mathcal{E}=\bigcup_{\ell=0}^{L-1}\bigcup_{i=0}^{\ell-1}T^{-(\ell-i-1)}\psi_1V_{n+i,K}
 =\bigcup_{k=0}^{L-1}T^{-(k-1)}\bigcup_{i=0}^{L-k}\psi_1V_{n+i,K}
  =\bigcup_{k=0}^{L-1}T^{-(k-1)}\psi_1V_{n,K+L-k}
$$
which can be estimated as follows:
$$
\mu(\mathcal{E})
\lesssim\sum_{k=0}^{L-1}\frac{K+L-k}{n^{\frac1\alpha+1}}
\lesssim\frac{L(K+L)}{n^{\frac1\alpha+1}}.
$$
Since 
$$
\frac{\mu(V_{n,K+L-1})}{(K+L-1)n^{-\frac1\alpha}}
$$
is bounded and bounded away from $0$ uniformly in $n$ we get the desired estimates.
 \end{proof}

\begin{lem}
If $\alpha,\beta\in(0,1)$ then
$$
 \mathfrak{G}_{L,n}
 =\sum_{j=L}^\infty \frac{ \mu(V_{n,K}\cap T^{-j}\{\hat{Z}^L\ge1\}\cap G_j^c)}{\mathbb{P}(\hat{Z}^L\ge 1)}
 \longrightarrow 0
 $$
 as $n\to\infty$.
 \end{lem}

\begin{proof}
In order to verify Assumption~(II-iv) we write
$$
V_{n,K}\cap T^{-j}\{\hat{Z}^L \ge1\}\cap G_j^c
=V_{n,K}\cap T^{-j}\bigcup_{\ell=0}^{L-1}T^{-\ell}V_{n,K}\cap T^{-q}U_q
=V_{n,K}\cap T^{-(j-q)}\hat{\mathcal{V}}
$$
(as $G_j^c=T^{-q}U_q$), where
$$
\hat{\mathcal{V}}=\bigcup_{\ell=0}^{L-1}\!\left(V_{n+\ell+q,K}
\cup\bigcup_{i=0}^{\ell+q-1-i}T^{-(\ell+q-1-i)}\psi_1V_{n+i,K}\cap U_q\right)
$$
Since
$$
T^{-(\ell+q-1-i)}\psi_1V_{n+i,K}\cap U_q\subset A_{\ell+q-1-i}\cap U_q=\varnothing
$$
if $\ell+q-i-1<q$, or, if $i>\ell-1$,
we can put
$$
\hat{\mathcal{V}}'
=\bigcup_{\ell=0}^{L-1}\bigcup_{i=0}^{\ell-1}T^{-(\ell+q-1-i)}\psi_1V_{n+i,K}\cap U_q
$$
which implies $\hat{\mathcal{V}}'\subset V_{q,L-1}$ and to write
$$
\hat{\mathcal{V}}=\bigcup_{\ell=0}^{L-1} V_{n+\ell+q,K}\cup\hat{\mathcal{V}}'
=V_{n+q,K+L-1}\cup\hat{\mathcal{V}}'.
$$
Note that since
$$
T^{-(\ell+q-1-i)}\psi_1V_{n+i,K}\cap U_q
=\bigcup_{k=q}^{\ell+q-i-1} \psi_0^k\psi_1T^{-(\ell+q-i-2-k)}\psi_1V_{n+i,K}
$$
(the term $k=\ell+q-1-i$ is when one of the maps $\psi_1$ is deleted)
we get for the Lebesgue measure $m$
\begin{eqnarray*}
m\!\left(T^{-(\ell+q-1-i)}\psi_1V_{n+i,K}\cap U_q\right)
&\lesssim&\sum_{k=q}^{\ell+q-i-1}m\!\left(\psi_0^k\psi_1T^{-(\ell+q-i-2-k)}\psi_1V_{n+i,K}\right)\\
&\lesssim&\sum_{k=q}^{\ell+q-i-1}\frac1{k^{\frac1\alpha+1}}m(V_{n+i,K})\\
&\lesssim&\frac{K}{(n+i)^{\frac1\alpha+1}q^{\frac1\alpha}}
\end{eqnarray*}
as $D\psi_0^k|_{V_0}\lesssim k^{-(\frac1\alpha+1)}$.
Consequently
$$
m(\hat{\mathcal{V}}')
\lesssim \frac{KL^2}{n^{\frac1\alpha+1}q^{\frac1\alpha}}
$$
and therefore
\begin{eqnarray*}
m(\hat{\mathcal{V}})
&\lesssim&m(V_{n+q,K+L-1})+m(\hat{\mathcal{V}}')\\
&\lesssim &\frac{K+L}{(n+q)^{\frac1\alpha+1}}+\frac{KL^2}{n^{\frac1\alpha+1}q^{\frac1\alpha}}\\
&\lesssim &\frac{KL^2}{(n+q)^{\frac1\alpha+1}}.
\end{eqnarray*}
Since $h|_{V_{n,K}}\lesssim n$ one has
$$
\mu(V_{n,K}\cap T^{-j}\{\hat{Z}^L\ge1\}\cap G_j^c)
\lesssim n\cdot m(V_{n,K}\cap T^{-(j-q)}\hat{\mathcal{V}})
$$
and therefore
$$
V_{n,K}\cap T^{-(j-q)}\hat{\mathcal{V}}
=(V_{n,K}\cap \psi_0^{j-q}\hat{\mathcal{V}})
\cup\bigcup_{i=0}^{j-q-1}\bigcup_{k=q}^{j-q-i-1}\psi_0^k\psi_1T^{-(j+q-i-2-k)}\psi_1\psi_0^i\hat{\mathcal{V}}\cap V_{n,K}.
$$
Since $\psi_0^k\psi_1T^{-(j+q-i-2-k)}\psi_1\psi_0^i\hat{\mathcal{V}}\subset A_k$ 
we must have $n\le k<n+K$ which implies in particular that $j-q-i-2\ge n$ and thus 
$j\ge n+1$ in the second term. 
For $k\in[n,n+K)$ this leads to
\begin{eqnarray*}
m(\psi_0^k\psi_1T^{-(j+q-i-2-k)}\psi_1\psi_0^i\hat{\mathcal{V}}\cap V_{n,K})
&\lesssim&\frac1{k^{\frac1\alpha+1}}\left(\frac{q}{q+i}\right)^{\frac1\alpha+1}m(\hat{\mathcal{V}})\\
&\lesssim&\frac1{n^{\frac1\alpha+1}}\left(\frac{q}{q+i}\right)^{\frac1\alpha+1}\frac{KL^2}{(n+q)^{\frac1\alpha+1}}
\end{eqnarray*}
as $D\psi_0^i|_{V_{n,K}}\lesssim\left(\frac{q}{q+i}\right)^{\frac1\alpha+1}$.
Since $h|_{V_{n,K}}\lesssim n$ this leads to
\begin{eqnarray*}
\mu\!\left(\bigcup_{i=0}^{j-q-1}\bigcup_{k=q}^{j-q-i-1}\psi_0^k\psi_1T^{-(j+q-i-2-k)}\psi_1\psi_0^i\hat{\mathcal{V}}\cap V_{n,K}
\right)\hspace{-5cm}\\
&\lesssim&n\sum_{i=0}^{j-q-1}K\frac1{n^{\frac1\alpha+1}}\left(\frac{q}{q+i}\right)^{\frac1\alpha+1}\frac{KL^2}{(n+q)^{\frac1\alpha+1}}\\
&\lesssim&\frac{K^2L^2}{n^{\frac1\alpha}}\frac{q^{\frac1\alpha+1}}{q^{\frac1\alpha}}\frac1{(n+q)^{\frac1\alpha+1}}\\
&=&\frac{K^2L^2q}{n^{\frac1\alpha}(n+q)^{\frac1\alpha+1}}.
\end{eqnarray*}
For the other term we get 
$$
\mu(V_{n,K}\cap\psi_0^{j-q}\hat{\mathcal{V}})
\lesssim n\cdot m(V_{n,K}\cap\psi_0^{j-q}\hat{\mathcal{V}})
\lesssim n\!\left(m(V_{n,K}\cap V_{n+j,K+L-1})+
m(V_{n,K}\cap \psi_0^{j-q}\hat{\mathcal{V}}')\right).
$$
Since for $j\ge L$ one has $V_{n,K}\cap V_{n+j,K+L-1}=\varnothing$ and since
$\hat{\mathcal{V}}'\subset V_{q,L-1}$ one gets
$\psi_0^{j-q}\hat{\mathcal{V}}'\subset V_{j,L-1}$ and therefore
$V_{n,K}\cap \psi_0^{j-q}\hat{\mathcal{V}}'\not=\varnothing$ only if $n-L\le j<n+K$
we obtain
$$
m(V_{n,K}\cap\psi_0^{j-q}\hat{\mathcal{V}}')
\lesssim \frac{K+L}{n^{\frac1\alpha+1}n^{\beta\frac1\alpha}}
$$
 as $q=j^\beta$.
 
 Finally we arrive at 
 $$
 \mu(V_{n,K}\cap T^{-j}\{\hat{Z}^L\ge1\}\cap G_j^c)
 \lesssim\frac{k^2L^2j^\beta}{n^{\frac1\alpha}(n+j^\beta)^{\frac1\alpha+1}}\chi_{[n,\infty)}(j)
 +\frac{k+L}{n^{\frac1\alpha}n^{\frac\beta\alpha}}\chi_{[n-L,n+K)}(j)
 $$
 and since $\mathbb{P}(\hat{Z}^L\ge 1)\gtrsim LKn^{-\frac1\alpha}$ this gives us
 $$
 \mathfrak{G}_{L,n}
 \lesssim \sum_{j=L}^\infty\frac{j^\beta}{(n+j^\beta)^{\frac1\alpha+1}}\chi_{[n,\infty)}(j)
 +\frac1{n^{\frac\beta\alpha}}
 \longrightarrow0
 $$
 as $n\to\infty$.
 \end{proof}

The estimates of the $\mathcal{R}$ terms are much the same as for non-parabolic points:
$$
\mathcal{R}_{1,\mbox{\tiny corr}}
\lesssim \frac{a^L}{\rho^w}\frac1{\Delta^\gamma}
\lesssim a^L\rho^{-w}N^{-v\gamma},
$$
where we put $\Delta=N^v$ for some $v<1$.
Similarly
$$
\mathcal{R}_{1,\mbox{\tiny ann}}
\lesssim \rho^w\mu(\mathcal{V})h(a_n)
\lesssim n\rho^w\mu(\mathcal{V}).
$$
No change in the $\mathcal{R}_3$ term as we still have $\mathcal{R}_3\lesssim\mu(\mathcal{V})^{1-v}$.
Similarly we have
 $$
 \mathcal{R}_2
 \lesssim \frac1{L'^{\frac1{\alpha}-2\beta}}+\mathfrak{G}_{L',n}.
 $$

 For the coefficients of the limiting compound Poisson distribution we get if $\ell\ge2$:
 $$
 \lambda_\ell=\lim_{L\to\infty}\lim_{n\to\infty}\frac{\mathbb{P}(\hat{Z}^L=\ell)}{\mathbb{P}(\hat{Z}^L\ge1)}
 =\lim_{L\to\infty}\lim_{n\to\infty}\frac{\mu(V_{n,\ell})}{\mu(V_{n,L-K-1})}(1+o(1))
 =0
 $$
 which implies that we get in the limit that 
 $\lim_{L\to\infty}\lim_{n\to\infty}\frac{\mathbb{P}(Z^N=0)}{\mathbb{P}(Z^L\ge1)}\longrightarrow e^{t}$,
 that is the appropriately rescaled entry times converge to an exponential distribution.

Notice that for the parabolic point we have the non-Kac scaling 
$N=\frac{t}{\mathbb{P}(\hat{Z}^L\ge1)/L}\sim\frac{tL}{\mu(V_{n,L})}\sim tn^{\frac1\alpha}$
rather than the standard Kac scaling which would require $N$ to have the value
$\frac{t}{\mu(U_n)}\sim tn^{\frac1\alpha-1}$ which grows much more slowly.

An application of Theorem~\ref{helper.theorem} then leads to the following result.

\begin{thrm} Let $T: I\circlearrowleft$ be as described above,
where the map $T$ is the parabolic map $T_{\alpha}$.
Assume $0<\alpha<\frac13$.
Denote by $\mu$ the absolutely continuous invariant measure, then for all $t>0$ the counting function 
$$
W_{x,\rho}=\sum_{j=0}^{N-1}\ind_{B_\rho(x)}\circ T^j
$$
converges in distribution to Poisson($t$) 
 for Lebesgue almost every $x\in [0,1]$ and for $x=0$, where
for:\\
 $x\not=0$: $N=\frac{t}{\mathbb{P}(Z_{B_\rho(x)}^L\ge 1)/L}$\\
$x=0$: $N=\frac{t}{\mu(V_{n,L+K})/L}$ where $K<L$ (e.g.\ $K=L^\eta$ for some $\eta<1$)
and a double limit $n\to\infty$ and then $L\to\infty$.
\end{thrm}

\end{document}